\documentclass[12pt, a4paper]{article}
\usepackage{latexsym} 
\usepackage{amsmath,amssymb}
\usepackage{mathtools}
\usepackage{bm} 
\usepackage[dvipdfmx]{graphicx}
\usepackage{float}
\usepackage[T1]{fontenc}
\usepackage{textcomp}
\usepackage{type1cm}
\usepackage{bbm} 
\usepackage{pifont} 
\usepackage{multicol} 
\usepackage{amsthm} 
\usepackage{ascmac} 
\usepackage{enumerate}
\usepackage{tikz}
\usepackage{tikz-cd}
\usetikzlibrary{arrows} 
\usepackage{mathrsfs} 
\usepackage{url}
\usepackage{alltt} 
\usepackage{braket} 
\usepackage{cases} 
\usepackage[]{multicol} 
\usepackage{amssymb}
\usepackage{color, soul}
\usepackage{empheq}

\theoremstyle{definition}

\newtheorem{thm}[subsection]{Theorem}
\newtheorem{cor}[subsection]{Corollary}
\newtheorem{conj}[subsection]{Conjecture}

\newtheorem{prop}[subsection]{Proposition}

\newtheorem{lem}[subsection]{Lemma}
\newtheorem*{thm*}{Theorem}

\newcommand{\Rarrow}{\Longrightarrow}

\newcommand{\LRarrow}{\Longleftrightarrow}

\DeclareMathOperator{\Z}{\mathbb{Z}}
\DeclareMathOperator{\Q}{\mathbb{Q}}

\usepackage{indentfirst}
\usepackage{setspace}

\begin{document}
\title{Some congruences of calibers of real quadratic fields}
\author{Naoto Fujisawa\thanks{Graduate School for Mathematics, Kyushu University, Motooka 744, Nisiku, Fukuoka, 819-0395, Japan\\
Email address: fujisawa.naoto.306@gmail.com}}
\date{}
\maketitle

\begin{abstract}
In this paper, the congruence equations for caliber and m-caliber in various discriminants are proven. Additionally, We also obtained the lengths of the periods of several continued fractions as corollaries.
\end{abstract}


\setlength{\baselineskip}{18pt}
\parskip=10pt plus 5pt

\section*{ \S 1. Introduction}
\setcounter{section}{1} 

A real quadratic number $w$ is $reduced$ if it satisfies $w>1$ and $-1<w'<0$,  where $w'$ is the algebraic conjugate of $w$ over the rational number field $\Q$. A quadratic number $w$ is reduced if and only if its usual continued fraction expansion is purely periodic. Let $\mathcal{Q}(D)$ be the set of all reduced quadratic numbers of a given discriminant $D$:
\[
\mathcal{Q}(D):=\{w \mid \mathrm{disc}(w)=D,\ w:\mathrm{reduced}\}. 
\]
Here, the discriminant $D$ of a real quadratic number $w$, denoted $D=$disc($w$), is the quantity $D=b^2-4ac$, if the quadratic equation of $w$ is
\[
aw^2+bw+c=0 \ (a,b,c \in \Z, \ a>0, \ \mathrm{GCD}(a,b,c)=1).
\]
The set $\mathcal{Q}(D)$ is finite and its cardinality is denoted by $\kappa(D)$. When $D$ is a fundamental discriminant, the number $\kappa(D)$ is called the $caliber$ of $\Q(\sqrt{D})$.\par

We write $w_1 \sim w_2$ if the two numbers $w_1$ and $w_2$ are $\mathrm{GL}_2(\Z)$-equivalent under the linear fractional transformation. It is known that $w_1 \sim w_2$ if
and only if their periods of continued fraction expansions are cyclically equivalent. Let $\mathcal{R}(D)$ be the set of $\mathrm{GL}_2(\Z)$-equivalence classes of $\mathcal{Q}(D)$ and $h(D)$ be its cardinality:
\[
\mathcal{R}(D)=\mathcal{Q}(D)/\sim, \ h(D)=\sharp\mathcal{R}(D).
\]
The number $h(D)$ is the wide class number of discriminant $D$.\par

A real quadratic number $w$ is $m$-$reduced$ if $w >1$ and $0 < w' < 1$. A number is $m$-reduced if and only if its ``minus'' continued fraction expansion is purely periodic. Let $\mathcal{Q}^+(D)$ be the set of all $m$-reduced numbers of a given discriminant $D$:
\[
\mathcal{Q}^+(D):=\{w \mid \mathrm{disc}(w)=D, \ w:m\text{-reduced}\}.
\]
This is also a finite set and its cardinality will be denoted by $\kappa^+(D)$. We say two numbers $w_1$ and $w_2$ are strictly equivalent, written $w_1$$\approx$ $w_2$, if the two are related with each other by a transformation in $\mathrm{SL}_2(\Z)$. Two elements in $\mathcal{Q}^+(D)$ are strictly equivalent if and only if the periods of their minus continued fraction expansions are cyclically equivalent. Let $\mathcal{R}^+(D)$ be the set of $\mathrm{SL}_2(\Z)$-equivalence classes of $\mathcal{Q}^+(D)$ and $h^+(D)$ be its cardinality:
\[
\mathcal{R}^+(D)=\mathcal{Q}^+(D)/\approx, \ h^+(D)=\#\mathcal{R}^+(D).
\]
The number $h^+(D)$ is the narrow class number of discriminant of $D$.\par

We denote by $\varepsilon_D=\frac{t_D+u_D\sqrt{D}}{2}$ the fundamental unit of $\Z[\frac{D+\sqrt{D}}{2}]$ and $N(\varepsilon_D)$ its norm.\par

In this paper, we show the following theorems.

\begin{thm*}{\ref{8p}}
Let $p$ be a prime number such that $p \equiv 1$ (mod 4) and let $x_p$ and $y_p$ be integers satisfying $p=x_p^2+y_p^2$ and $0<x_p<y_p$. Then we have 
\[
\kappa^+(8p) \equiv 1-(-1)^{x_p}\ \mathrm{(mod\ 4)}.
\]
\end{thm*}

\begin{thm*}{\ref{pq}}
Let $p$ and $q$ ($p<q$) be prime numbers such that $p \equiv q \equiv 3$ (mod 4). Then we have 
\[
\kappa^+(pq) \equiv 1-\left(\frac{q}{p}\right)\ \mathrm{(mod\ 4)}.
\]
\end{thm*}

\begin{thm*}{\ref{mcal4pq}}
Let $p$ and $q$ ($p<q$) be prime numbers such that $p \equiv q \equiv 3$ (mod 4). Then we have 
\[
\kappa^+(4pq) \equiv 2\ \mathrm{(mod\ 4)}.
\]
\end{thm*}

\begin{thm*}{\ref{pand4p}}
Let $p$ be a prime number such that $p \equiv 1$ (mod 4). Then we have
\[
p \equiv 1\ \mathrm{(mod\ 8)} \Rarrow \kappa(p) \equiv \kappa(4p)+2\ \mathrm{(mod\ 4)},
\]
\[
p \equiv 5\ \mathrm{(mod\ 8)} \Rarrow \kappa(p) \equiv \kappa(4p)\ \mathrm{(mod\ 4)}.
\]
\end{thm*}

\begin{thm*}{\ref{cal8p}}
Let $p$ be a prime number such that $p \equiv 1$ (mod 4). Then we have 
\[
p \equiv 1\ \mathrm{(mod\ 8)} \Rarrow \kappa(8p) \equiv 2\ \mathrm{(mod\ 4)},
\]
\[
p \equiv 5\ \mathrm{(mod\ 8)} \Rarrow \kappa(8p) \equiv 0\ \mathrm{(mod\ 4)}.
\]

\end{thm*}

\begin{thm*}{\ref{4pand8p}}
Let $p$ be a prime number such that $p \equiv 3$ (mod 4). Then we have
\[
p \equiv 3\ \mathrm{(mod\ 8)} \Rarrow \kappa(4p) \equiv \kappa(8p) \equiv 2 \ \mathrm{(mod\ 4)},
\]
\[
p \equiv 7\ \mathrm{(mod\ 8)} \Rarrow \kappa(4p) \equiv \kappa(8p) \equiv 0 \ \mathrm{(mod\ 4)}.
\]
\end{thm*}

\begin{thm*}{\ref{9p}}
Let $p$ be a prime number such that $p \equiv 1$ (mod 4). Then we have
\[
p=5\ \mathrm{or}\ p \equiv 1\ \mathrm{(mod\ 3)} \Rarrow \kappa(9p) \equiv 2\ \mathrm{(mod\ 4)},
\]
\[
p \equiv 2\ \mathrm{(mod\ 3)}\ (p \not=5)\ \Rarrow \kappa(9p) \equiv 0\ \mathrm{(mod\ 4)}.
\]
\end{thm*}

\begin{thm*}{\ref{16p}}
Let $p>2$ be a prime number. Then we have
\[
p=3\ \mathrm{or}\ p \equiv 1\ \mathrm{(mod\ 4)} \Rarrow \kappa(16p) \equiv 2\ \mathrm{(mod\ 4)},
\]
\[
p \equiv 3\ \mathrm{(mod\ 4)} \ (p \not=3)\ \Rarrow \kappa(16p) \equiv 0\ \mathrm{(mod\ 4)}.
\]
\end{thm*}

\begin{thm*}{\ref{calpq}}
Let $p$ and $q$ ($p<q$) be prime numbers such that $p \equiv q \equiv 3$ (mod 4). Then we have 
\[
\kappa(pq) \equiv 1+\left(\frac{q}{p}\right)\ \mathrm{(mod\ 4)}.
\]
\end{thm*}

\begin{thm*}{\ref{cal4pq}}
Let $p$ and $q$ ($p<q$) be prime numbers such that $p \equiv q \equiv 3$ (mod 4). Then we have 
\[
\kappa(4pq) \equiv 1+\left(\frac{q}{p}\right)\ \mathrm{(mod\ 4)}.
\]
\end{thm*}

\begin{thm*}{\ref{4pq}}
Let $p$ and $q$ be prime numbers such that $p \equiv 1$ (mod 4) and $q \equiv 3$ (mod 4). Then we have 
\[
\kappa(4pq) \equiv 0\ \mathrm{(mod\ 4)}.
\]
\end{thm*}

\section*{ \S 2. Preliminaries}
 \setcounter{section}{2}
 \setcounter{subsection}{0}
Suppose the continued fraction expansion of $\alpha$ and the minus continued fraction expansion of $\beta$ are both purely periodic: 
\[
\alpha=\overline{[a_1,a_2,...,a_n]}=a_1+\cfrac{1}{a_2+\cfrac{1}{\ddots+\cfrac{1}{a_n+\cfrac{1}{a_1+\cfrac{1}{\ddots}}}}}
\]
and
\[
\beta=\overline{[[b_1,b_2,...,b_m]]}=b_1-\cfrac{1}{b_2-\cfrac{1}{\ddots-\cfrac{1}{b_m-\cfrac{1}{b_1-\cfrac{1}{\ddots}}}}}.
\]
We define $l(\alpha)\coloneqq n$ and $l^+(\beta)\coloneqq m$ to be their minimum period lengths and 
\[
S(\alpha) \coloneqq \sum_{i=1}^n a_i \ \ \ \mathrm{and} \ \ \ S^+(\beta) \coloneqq \sharp \{j \mid 1 \leq j \leq m, \ b_j \geq 3\}.
\]
We need several propositions and lemmas to prove the theorems. For proofs, we refer the reader to \cite{cox} and \cite{mori} except for Lemma \ref{equivalent}
\begin{prop}[\cite{mori}, Proposition 2.1]\label{equation}
We have 
\[
\kappa(D) = \sum_{[\alpha] \in \mathcal{R}(D)}l(\alpha) \ \ \ \mathrm{and} \ \ \ \kappa^+(D)=\sum_{[\alpha] \in \mathcal{R}(D)}S(\alpha).
\]
\end{prop}
\begin{lem}[\cite{mori}, Lemma 2.2]\label{number}
Let
\[
E=
\left(
\begin{matrix}
0&1\\
1&0
\end{matrix}
\right)
\ \ \ \mathrm{and} \ \ \ O=
\left(
\begin{matrix}
1&1\\
1&0
\end{matrix}
\right)
\]
be elements in $\mathrm{GL}_2(\mathbb{F}_2)$. Consider the product
\[
M=M_1...M_n
\]
of length $n$ of $k$ $O$'s  and $(n-k)$ $E$'s. Then we have 
\[
n \equiv k \ \mathrm{(mod\ 2)} \LRarrow M \in \left\{
I=\left(
\begin{matrix}
1&0\\
0&1
\end{matrix}
\right), 
O=
\left(
\begin{matrix}
1&1\\
1&0
\end{matrix}
\right),
O^2=
\left(
\begin{matrix}
0&1\\
1&1
\end{matrix}
\right)
\right\}.
\]
\end{lem}

\begin{lem}[\cite{mori}, Lemma 2.3]\label{parity}
Let $\alpha \in \mathcal{Q}(D)$.
\begin{enumerate}
\setlength{\parskip}{-18pt}
\setlength{\itemsep}{18pt}
\item If $D$ is odd, then we have $l(\alpha) \equiv S(\alpha)$ (mod 2).
\item If $D$ is even and $N(\varepsilon_D)=-1$, then $S(\alpha)$ is even. 
\end{enumerate}
\end{lem}

\begin{prop}[\cite{cox}]\label{classnumber}
Let $D=f^2D_0$ be a discriminant, where $D_0$ is a fumdamental discriminant and $f$ is a positive integer. Let $\mu$ be an integer such that $\varepsilon_D=\varepsilon_{D_0}^{\mu}$. Then, we have
\[
h(D)=\frac{h(D_0)f}{\mu}\prod_{p \mid f}\left(1-\frac{\chi_{D_0}(p)}{p}\right)
\]
where $p$ runs over prime factors of $f$ and $\chi_{D_0}$ is the Kronecker character of $\Q(\sqrt{D_0})$.
\end{prop}

\begin{lem}\label{equivalent}
Suppose $\alpha \in \mathcal{Q}(D)$ satisfies $\alpha \sim -\frac{1}{\alpha'}$. Then the continued fraction of $\alpha$ has the form
\[
\alpha \sim \overline{[c_1,c_2,...,c_{l-1},c_{l},c_{l-1},...,c_2,c_1]}
\]
if $N(\varepsilon_D)=-1$ and 
\[
\alpha \sim \overline{[c_1,...,c_{l},c_{l},...,c_1]}\ \ \  \mathrm{or}\ \  \ \overline{[c_0,c_1,...,c_l,c_{l+1},c_{l},...,c_1]}
\]
if $N(\varepsilon_D)=1$. In particular, let $ax^2+bx+c\ (a>0)$ be the minimal polynomial of $\alpha$. Then, if 
$\alpha=\overline{[c_1,c_2,...,c_{l-1},c_{l},c_{l-1},...,c_2,c_1]}$ or $\overline{[c_1,...,c_{l},c_{l},...,c_1]}$, the equality $a=-c$ holds, and if $\alpha = \overline{[c_0,c_1,...,c_l,c_{l+1},c_{l},...,c_1]}$, $a$ divides $b$. 
 \end{lem}

\begin{proof}
We prove the assertion for $N(\varepsilon_D)=1$. The other case is similarly proved. When $\alpha \sim -\frac{1}{\alpha'}$, there exists an integer $i \in \{1,...,n\}$ such that
\[
\overline{[a_1,...,a_n]}=\overline{[a_i,a_{i-1},...,a_1,a_n,...,a_{i+1}]}.
\]
holds. 
Therefore, 
\[
\alpha = \overline{[c_1,...,c_{l},c_{l},...,c_1]}
\]
when $i$ is even, and 
\[
\alpha = \overline{[c_0,c_1,...,c_l,c_{l+1},c_{l},...,c_1]}
\]
when $i$ is odd. \par
If $\alpha=\overline{[c_1,...,c_{l},c_{l},...,c_1]}$, then $\alpha=-\frac{1}{\alpha'}$ holds. Hence, we have
\[
a\left(-\frac{1}{\alpha'}\right)^2+b\left(-\frac{1}{\alpha'}\right)+c=0.
\]
Therefore, since $-cx^2+bx-a$ is a minimal polynomial  of $\alpha$, we get $a=-c$. When $\alpha = \overline{[c_0,c_1,...,c_l,c_{l+1},c_{l},...,c_1]}$ holds, we have
\[
\alpha=c_0+\frac{1}{-\frac{1}{\alpha'}}=c_0-\alpha'. 
\]
Since $-\frac{b}{a}=\alpha+\alpha'=c_0 \in \Z$, we get $a|b$.  
\end{proof}

\section*{ \S 3. Proofs of Theorems}
\setcounter{section}{3}
\setcounter{subsection}{0}

Theorems \ref{4pand8p}, \ref{9p}, \ref{16p}, \ref{calpq} and \ref{4pq} have almost identical proofs.

\begin{thm}\label{8p}
Let $p$ be a prime number such that $p \equiv 1$ (mod 4) and let $x_p$ and $y_p$ be integers satisfying $p=x_p^2+y_p^2$ and $0<x_p<y_p$. Then we have 
\[
\kappa^+(8p) \equiv 1-(-1)^{x_p}\ \mathrm{(mod\ 4)}.
\]
\end{thm}

\begin{proof}
When $N(\varepsilon_{8p})=-1$, this is proved in \cite{mori}. We will suppose $N(\varepsilon_{8p})=1$. Let $\alpha$, $\beta$, $\gamma$ and $\delta$ be the largest solutions of equations
\[
(x_p+y_p)x^2+2(x_p-y_p)x-(x_p+y_p)=0,
\]
\[
 (y_p-x_p)x^2-2(x_p+y_p)x+(x_p-y_p)=0,
\]
\[
x^2-2\lfloor\sqrt{2p}\rfloor x+\lfloor\sqrt{2p}\rfloor^2-2p=0,
\]
and
\[
 2x^2-4\left\lfloor\sqrt{\frac{p}{2}}\right\rfloor x+2\left\lfloor\sqrt{\frac{p}{2}}\right\rfloor^2-p=0
\]
respectively. (The form of continued fraction expansions of $\alpha$ and $\beta$ are $\overline{[a_1,...,a_n,a_n,...,a_1]}$ and $\gamma$ and $\delta$ are $\overline{[a_0,a_1,...,a_n,a_{n+1}, a_n,...,a_1]}$ from $N(\varepsilon_{8p})=1.$) By Lemma \ref{equivalent}, we may take representatives of $\mathcal{R}(8p)$ as 
\[
\alpha,\ \gamma, \ \gamma_1,\ -\frac{1}{\gamma_1'}, \ ..., \gamma_t, \ -\frac{1}{\gamma_t'}.
\] 
From Lemma \ref{parity}, we have
\[
S(\gamma_i) = S\left(-\frac{1}{\gamma_i'}\right) \equiv 0\ \mathrm{(mod\ 2)}.
\]
Hence, by Proposition \ref{equation}, it suffices to prove that
\[
S(\alpha)+S(\gamma) \equiv 1-(-1)^{x_p}\ \mathrm{(mod\ 4)}.
\]\par
First, we will show that $S(\alpha) \equiv 1-(-1)^{x_p}$ (mod 4). Let 
\[
\alpha=\overline{[a_1,...,a_n,a_n,...,a_1]}
\]
be a continued fraction expansion of $\alpha$ (Then $\beta=\overline{[a_{n},...,a_1,a_1,...,a_n]}.$) and set
\[
\left(\begin{matrix}
p&q\\
r&s
\end{matrix}\right)
=
\left(\begin{matrix}
a_1&1\\
1&0
\end{matrix}\right)
...
\left(\begin{matrix}
a_n&1\\
1&0
\end{matrix}\right)
,\ 
\left(\begin{matrix}
P&Q\\
R&S
\end{matrix}\right)
=
\left(\begin{matrix}
p&q\\
r&s
\end{matrix}\right)
\left(\begin{matrix}
p&r\\
q&s
\end{matrix}\right), 
\]
\[
\left(\begin{matrix}
P'&Q'\\
R'&S'
\end{matrix}\right)
=
\left(\begin{matrix}
p&r\\
q&s
\end{matrix}\right)
\left(\begin{matrix}
p&q\\
r&s
\end{matrix}\right).
\]
Then we have
\[
R\alpha^2+(S-P)\alpha-Q=0, \ R'\beta^2+(S'-P')\beta-Q'=0.
\]
Set $d=\mathrm{GCD}(R, S-P, Q)=\mathrm{GCD}(R', S'-P', Q')$. (The greatest common divisor $d$ is $u_{8p}$ from a classical fact of continued fractions.) Then we have
\[
R=d(x_p+y_p), \ S-P=2d(x_p-y_p), \ -Q=-d(x_p+y_p)
\]
and
\[
R'=d(y_p-x_p), \ S'-P'=-2d(x_p+y_p), \ -Q'=-d(y_p-x_p).
\]
Therefore we have
\[
\begin{cases}
pr+qs=d(x_p+y_p),\\
(p^2+q^2)-(r^2+s^2)=2d(y_p-x_p),\\
pq+rs=d(y_p-x_p),\\
(p^2+r^2)-(q^2+s^2)=2d(x_p+y_p).
\end{cases}
\]
From this, we get $p-s=r+q$, $(p+s)x_p=(r-q)y_p$ because $p>r$ holds. In particular, $p+s$ and $q+r$ have the same parities and we have $\left(\begin{matrix}
p&q\\
r&s
\end{matrix}\right)
\equiv
\left(\begin{matrix}
1&0\\
0&1
\end{matrix}\right)$ or 
$\left(\begin{matrix}
0&1\\
1&0
\end{matrix}\right)$ (mod 2).\par
When 
$\left(\begin{matrix}
p&q\\
r&s
\end{matrix}\right)
\equiv
\left(\begin{matrix}
1&0\\
0&1
\end{matrix}\right)$
(mod 2) and $n$ is even, noting that $p-s=r+q$, we have
\[
\left(\begin{matrix}
p&q\\
r&s
\end{matrix}\right)
\equiv
\left(\begin{matrix}
1&0\\
0&1
\end{matrix}\right)
,\ 
\left(\begin{matrix}
3&0\\
0&3
\end{matrix}\right)
,\ 
\left(\begin{matrix}
1&2\\
2&1
\end{matrix}\right)
,\ 
\left(\begin{matrix}
3&2\\
2&3
\end{matrix}\right)
\ \mathrm{(mod \ 4)},
\]
Since $(p+s)x_p=(r-q)y_p$, we obtain $x_p$ is even. From Lemma \ref{number}, we have $S(\alpha) \equiv 2$ (mod 4). Similarly when 
$\left(\begin{matrix}
p&q\\
r&s
\end{matrix}\right)
\equiv
\left(\begin{matrix}
0&1\\
1&0
\end{matrix}\right)$
(mod 2) and $n$ is even, we have
\[
\left(\begin{matrix}
p&q\\
r&s
\end{matrix}\right)
\equiv
\left(\begin{matrix}
0&3\\
1&0
\end{matrix}\right)
,\ 
\left(\begin{matrix}
2&3\\
1&2
\end{matrix}\right)
,\ 
\left(\begin{matrix}
0&1\\
3&0
\end{matrix}\right)
,\ 
\left(\begin{matrix}
2&1\\
3&2
\end{matrix}\right)
\ \mathrm{(mod \ 4)}.
\]
Then we see that $x_p$ is odd and $S(\alpha) \equiv 2$ (mod 4). When 
$\left(\begin{matrix}
p&q\\
r&s
\end{matrix}\right)
\equiv
\left(\begin{matrix}
1&0\\
0&1
\end{matrix}\right)$
(mod 2) and $n$ is odd, we have
\[
\left(\begin{matrix}
p&q\\
r&s
\end{matrix}\right)
\equiv
\left(\begin{matrix}
3&2\\
0&1
\end{matrix}\right)
,\ 
\left(\begin{matrix}
1&2\\
0&3
\end{matrix}\right)
,\ 
\left(\begin{matrix}
3&0\\
2&1
\end{matrix}\right)
,\ 
\left(\begin{matrix}
1&0\\
2&3
\end{matrix}\right)
\ \mathrm{(mod \ 4)}
\]
and we conclude $x_p$ is odd and $S(\alpha) \equiv 2$ (mod 4). Finally, when 
$\left(\begin{matrix}
p&q\\
r&s
\end{matrix}\right)
\equiv
\left(\begin{matrix}
0&1\\
1&0
\end{matrix}\right)$
(mod 2) and $n$ is odd, we have
\[
\left(\begin{matrix}
p&q\\
r&s
\end{matrix}\right)
\equiv
\left(\begin{matrix}
0&1\\
1&2
\end{matrix}\right)
,\ 
\left(\begin{matrix}
0&3\\
3&2
\end{matrix}\right)
,\ 
\left(\begin{matrix}
2&1\\
1&0
\end{matrix}\right)
,\ 
\left(\begin{matrix}
2&3\\
3&0
\end{matrix}\right)
\ \mathrm{(mod \ 4)}
\]
to conclude $x_p$ is even and $S(\alpha) \equiv 0$ (mod 4). As a result, we have $S(\alpha) \equiv 1-(-1)^{x_p}$ (mod 4). \par

Next, we will show that $S(w) \equiv 0$ (mod 4). Let 
\[
\gamma=\overline{[a_0,a_1,...,a_n,a_{n+1},a_{n},...,a_1]}
\]
and
\[
\delta=\overline{[a_{n+1},a_n,...,a_1,a_0,a_1,...,a_n]}
\]
be the continued fraction expansions and set
\[
\left(\begin{matrix}
p&q\\
r&s
\end{matrix}\right)
=
\left(\begin{matrix}
a_1&1\\
1&0
\end{matrix}\right)
...
\left(\begin{matrix}
a_n&1\\
1&0
\end{matrix}\right)
,
\]
\[
\left(\begin{matrix}
P&Q\\
R&S
\end{matrix}\right)
=
\left(\begin{matrix}
a_0&1\\
1&0
\end{matrix}\right)
\left(\begin{matrix}
p&q\\
r&s
\end{matrix}\right)
\left(\begin{matrix}
a_{n+1}&1\\
1&0
\end{matrix}\right)
\left(\begin{matrix}
p&r\\
q&s
\end{matrix}\right)
, \ 
\]
\[
\left(\begin{matrix}
P'&Q'\\
R'&S'
\end{matrix}\right)
=
\left(\begin{matrix}
a_{n+1}&1\\
1&0
\end{matrix}\right)
\left(\begin{matrix}
p&r\\
q&s
\end{matrix}\right)
\left(\begin{matrix}
a_0&1\\
1&0
\end{matrix}\right)
\left(\begin{matrix}
p&q\\
r&s
\end{matrix}\right).
\]\par
From a classical fact of continued fractions, we get
\[
\left(\begin{matrix}
P&Q\\
R&S
\end{matrix}\right)
=
\left(\begin{matrix}
\frac{t_{8p}+2\lfloor\sqrt{2p}\rfloor u_{8p}}{2}&(2p-\lfloor\sqrt{2p}\rfloor^2)u_{8p}\\
u_{8p}&\frac{t_{8p}-2\lfloor\sqrt{2p}\rfloor u_{8p}}{2}
\end{matrix}\right),
\]
\[
\left(\begin{matrix}
P'&Q'\\
R'&S'
\end{matrix}\right)
=
\left(\begin{matrix}
\frac{t_{8p}+4\left\lfloor\sqrt{\frac{p}{2}}\right\rfloor u_{8p}}{2}&(2\left\lfloor\sqrt{\frac{p}{2}}\right\rfloor^2-p)u_{8p}\\
2u_{8p}&\frac{t_{8p}-4\left\lfloor\sqrt{\frac{p}{2}}\right\rfloor u_{8p}}{2}
\end{matrix}\right).
\]
Since $u_{8p}$ is even, we see that $R$ and $Q$ are even and $S \equiv P$ (mod 4). Moreover, since $PS-QR=1$, S and $P$ are both odd. \par
Assume that $P \equiv S \equiv 1$ (mod 4) and $R \equiv 2$ (mod 4). By $R=a_{n+1}p^2+2pq \equiv 2$ (mod 4), $p$ is odd and we have $a_{n+1}+2q \equiv 2$ (mod 4). If $a_{n+1} \equiv 2$ (mod 4), $q$ is even because $R \equiv 2$ (mod 4). Then $s$ is odd from $ps-qr=(-1)^n$. Since $S = a_{n+1}pr+ps+qr \equiv 2r + (-1)^n \equiv 1$ (mod 4), the parities of $r$ and $n$ are the same. If $a_{n+1} \equiv 0$ (mod 4), then $q$ is odd. Moreover, since $ps-qr=(-1)^n$, the parities of $r$ and $s$ are different. As $S \equiv 2r+(-1)^n \equiv 1$ (mod 4), $r$ and $n$ have the same parities. Furthermore, as $R' = a_0p^2+2pr \equiv a_0+2r \equiv 0$ (mod 4), $r$ is even when $a_0 \equiv 0$ (mod 4) and $r$ is odd when $a_0 \equiv 2$ (mod 4).\par
In summary, we have
\[
\begin{pmatrix}
p&q\\
r&s
\end{pmatrix}
\equiv
\begin{cases}
\begin{pmatrix}
1&0\\
1&1
\end{pmatrix}
&(a_{n+1} \equiv 2\ \mathrm{(mod\ 4)}, a_{0} \equiv 2\ \mathrm{(mod\ 4)}, n\mathrm{:odd} )\\
\begin{pmatrix}
1&0\\
0&1
\end{pmatrix}
&(a_{n+1} \equiv 2\ \mathrm{(mod\ 4)}, a_{0} \equiv 0\ \mathrm{(mod\ 4)}, n\mathrm{:even})\\
\begin{pmatrix}
1&1\\
0&1
\end{pmatrix}
&(a_{n+1} \equiv 0\ \mathrm{(mod\ 4)}, a_{0} \equiv 0\ \mathrm{(mod\ 4)}, n\mathrm{:even})\\
\begin{pmatrix}
1&1\\
1&0
\end{pmatrix}
&(a_{n+1} \equiv 0\ \mathrm{(mod\ 4)}, a_{0} \equiv 2\ \mathrm{(mod\ 4)}, n\mathrm{:odd})\\
\end{cases}
\ \mathrm{(mod\ 2)}.
\]
Here, as in our previous discussion on $\alpha$, we see that $\frac{t_{8p}}{2}=\frac{P+S}{2} \equiv 3$ (mod 4) when $u_{8p}=d \equiv 2$ (mod 4). Therefore, the following two cases of the above satisfy the requirement:
\[
\begin{pmatrix}
p&q\\
r&s
\end{pmatrix}
\equiv
\begin{cases}
\begin{pmatrix}
1&0\\
1&1
\end{pmatrix}
&(a_{n+1} \equiv 2\ \mathrm{(mod\ 4)}, a_{0} \equiv 2\ \mathrm{(mod\ 4)}, n\mathrm{:odd})\\
\begin{pmatrix}
1&1\\
1&0
\end{pmatrix}
&(a_{n+1} \equiv 0\ \mathrm{(mod\ 4)}, a_{0} \equiv 2\ \mathrm{(mod\ 4)}, n\mathrm{:odd})\\
\end{cases}
\ \mathrm{(mod\ 2)}.
\]
By Lemma \ref{number}, we get $S(\gamma) \equiv 0$ (mod 4).\par
Similarly, when $P \equiv S \equiv 1$ (mod 4) and $R \equiv 0$ (mod 4), noting that, since $R=2R'$, $r=2q$ holds when $a_0=2a_{n+1}$ and $p+r=2q$ holds when $a_0=2a_{n+1}+2$, we have
\[
\begin{pmatrix}
p&q\\
r&s
\end{pmatrix}
\equiv
\begin{cases}
\begin{pmatrix}
1&0\\
1&1
\end{pmatrix}
&(a_{n+1} \equiv 0\ \mathrm{(mod\ 4)}, a_{0} \equiv 2\ \mathrm{(mod\ 4)}, n\mathrm{:even} )\\
\begin{pmatrix}
1&1\\
0&1
\end{pmatrix}
&(a_{n+1} \equiv 2\ \mathrm{(mod\ 4)}, a_{0} \equiv 0\ \mathrm{(mod\ 4)}, n\mathrm{:even})\\
\begin{pmatrix}
1&1\\
1&0
\end{pmatrix}
&(a_{n+1} \equiv 2\ \mathrm{(mod\ 4)}, a_{0} \equiv 2\ \mathrm{(mod\ 4)}, n\mathrm{:even})\\
\begin{pmatrix}
1&0\\
0&1
\end{pmatrix}
&(a_{n+1} \equiv 0\ \mathrm{(mod\ 4)}, a_{0} \equiv 0\ \mathrm{(mod\ 4)}, n\mathrm{:even})\\
\end{cases}
\ \mathrm{(mod\ 2)}
\]
as well as $S(\gamma) \equiv 0$ (mod 4).\par
When $P \equiv S \equiv 3$ (mod 4) and $R \equiv 2$ (mod 4), we have
\[
\begin{pmatrix}
p&q\\
r&s
\end{pmatrix}
\equiv
\begin{cases}
\begin{pmatrix}
1&0\\
0&1
\end{pmatrix}
&(a_{n+1} \equiv 2\ \mathrm{(mod\ 4)}, a_{0} \equiv 0\ \mathrm{(mod\ 4)}, n\mathrm{:odd} )\\
\begin{pmatrix}
1&1\\
0&1
\end{pmatrix}
&(a_{n+1} \equiv 0\ \mathrm{(mod\ 4)}, a_{0} \equiv 0\ \mathrm{(mod\ 4)}, n\mathrm{:odd})\\
\end{cases}
\ \mathrm{(mod\ 2)}
\]
and $S(\gamma) \equiv 0$ (mod 4).\par
Finally, when $P \equiv S \equiv 3$ (mod 4) and $R \equiv 0$ (mod 4), no matrix satisfies the condition. Therefore, we get $S(\gamma) \equiv 0$ (mod 4).\par
All these show that we have 
\[
\kappa^+(8p) \equiv S(\alpha)+S(\gamma) \equiv 1-(-1)^{x_p}\ \mathrm{(mod\ 4)}.
\]
\end{proof}

\begin{thm}\label{pq}
Let $p$ and $q$ ($p<q$) be prime numbers such that $p \equiv q \equiv 3$ (mod 4) . Then we have 
\[
\kappa^+(pq) \equiv 1-\left(\frac{q}{p}\right)\ \mathrm{(mod\ 4)}.
\]
\end{thm}

\begin{proof}
Let $\alpha$ and $\beta$ be the largest solutions of equations
\[
x^2-Bx-\frac{1}{4}(pq-B^2)=0 \left(B=
\begin{cases}
\lfloor\sqrt{pq}\rfloor\ &(\lfloor\sqrt{pq}\rfloor \mathrm{:odd})\\
\lfloor\sqrt{pq}\rfloor-1\ &(\lfloor\sqrt{pq}\rfloor \mathrm{:even})
\end{cases}
\right)
\]
and
\[
px^2-pB'x-\frac{1}{4}(q-p{B'}^2)=0\left(B'=
\begin{cases}
\left\lfloor\sqrt{\frac{q}{p}}\right\rfloor\ &\left(\left\lfloor\sqrt{\frac{q}{p}}\right\rfloor \mathrm{:odd}\right)\\
\left\lfloor\sqrt{\frac{q}{p}}\right\rfloor-1\ &\left(\left\lfloor\sqrt{\frac{q}{p}}\right\rfloor \mathrm{:even}\right)
\end{cases}
\right)
\]
respectively. It suffices to prove that
\[
S(\alpha) \equiv 1-\left(\frac{q}{p}\right)\ \mathrm{(mod\ 4)}. 
\]\par
Let
\[
\alpha=\overline{[a_0,a_1,...,a_{n},a_{n+1},a_{n},...,a_1]}
\]
and
\[
\beta = \overline{[a_{n+1},a_{n},...,a_1,a_0,a_1,...,a_{n-1},a_n]}
\]
be continued fraction expansions of $\alpha$ and $\beta$ respectively. We define
\[
\left(\begin{matrix}
p&q\\
r&s
\end{matrix}\right)
=
\left(\begin{matrix}
a_1&1\\
1&0
\end{matrix}\right)
...
\left(\begin{matrix}
a_n&1\\
1&0
\end{matrix}\right)
, \ 
\left(\begin{matrix}
P&Q\\
R&S
\end{matrix}\right)
=
\left(\begin{matrix}
a_0&1\\
1&0
\end{matrix}\right)
\left(\begin{matrix}
p&q\\
r&s
\end{matrix}\right)
\left(\begin{matrix}
a_{n+1}&1\\
1&0
\end{matrix}\right)
\left(\begin{matrix}
p&r\\
q&s
\end{matrix}\right)
, \ 
\]
\[
\left(\begin{matrix}
P'&Q'\\
R'&S'
\end{matrix}\right)
=
\left(\begin{matrix}
a_{n+1}&1\\
1&0
\end{matrix}\right)
\left(\begin{matrix}
p&r\\
q&s
\end{matrix}\right)
\left(\begin{matrix}
a_0&1\\
1&0
\end{matrix}\right)
\left(\begin{matrix}
p&q\\
r&s
\end{matrix}\right).
\]\par
Assume that $u_{pq}$ is even. Then, from $u_{pq}=a_{n+1}p^2+2pq$, $p$ is even, $q$ and $r$ are odd and $u_{pq} \equiv 0$ (mod 8) holds. By the proof of Theorem \ref{mcal4pq}, we get $t_{pq} \equiv 6$ (mod 8). Therefore, we get $a_{n+1}pr+ps+qr \equiv \frac{t_{pq}-a_0u_{pq}}{2} \equiv 3$ (mod 4) and  $a_{n+1}r^2+a_0(a_{n+1}pr+ps+qr)+2rs \equiv a_{n+1}-a_0+2s \equiv 0$ (mod 4). Assume that $u_{pq}$ is odd. Then, $p$ is odd and $u_{pq} \equiv a_{n+1}+2q$ (mod 4) holds. From $-u_{pq} \equiv a_{0}+2r$ (mod 4), $a_{0}+a_{n+1}+2(r+q) \equiv 0$ (mod 4) holds. \par
In summary, we have
\[
\begin{pmatrix}
p&q\\
r&s
\end{pmatrix}
\equiv
\begin{cases}
\begin{pmatrix}
0&1\\
1&1
\end{pmatrix}
&(u_{pq}\mathrm{:even}, \ a_0 +a_{n+1} \equiv 0\ \mathrm{(mod\ 4)})\\
\begin{pmatrix}
0&1\\
1&0
\end{pmatrix}
&(u_{pq}\mathrm{:even}, \ a_0 +a_{n+1} \equiv 2\ \mathrm{(mod\ 4)})\\
\begin{pmatrix}
1&0\\
0&1
\end{pmatrix}
&(u_{pq}\mathrm{:odd}, \ a_0 +a_{n+1} \equiv 0\ \mathrm{(mod\ 4)})\\
\begin{pmatrix}
1&1\\
1&0
\end{pmatrix}
&(u_{pq}\mathrm{:odd}, \ a_0 +a_{n+1} \equiv 0\ \mathrm{(mod\ 4)})\\
\begin{pmatrix}
1&1\\
0&1
\end{pmatrix}
&(u_{pq}\mathrm{:odd}, \ a_0 +a_{n+1} \equiv 2\ \mathrm{(mod\ 4)})\\
\begin{pmatrix}
1&0\\
1&1
\end{pmatrix}
&(u_{pq}\mathrm{:odd}, \ a_0 +a_{n+1} \equiv 2\ \mathrm{(mod\ 4)})\\
\end{cases}
\ \mathrm{(mod\ 2)}.
\]
Therefore, from Lemma \ref{number} and Theorem \ref{calpq}, we have $S(\alpha) \equiv 1-\left(\frac{q}{p}\right)$ (mod 4).

\end{proof}

\begin{thm}\label{mcal4pq}
Let $p$ and $q$ ($p<q$) be prime numbers such that $p \equiv q \equiv 3$ (mod 4). Then we have 
\[
\kappa^+(4pq) \equiv 2\ \mathrm{(mod\ 4)}.
\]
\end{thm}

\begin{proof}
Let $\alpha$ and $\beta$ be the largest solutions of equations
\[
x^2-2\lfloor\sqrt{pq}\rfloor x+\lfloor\sqrt{pq}\rfloor^2-pq=0
\]
and
\[
px^2-2p\left\lfloor\sqrt{\frac{q}{p}}\right\rfloor x-p\left\lfloor\sqrt{\frac{q}{p}}\right\rfloor^2-q=0
\]
respectively. It suffices to prove that
\[
S(\alpha) \equiv 2\ \mathrm{(mod\ 4)}. 
\]\par
Let
\[
\alpha=\overline{[a_0,a_1,...,a_{n},a_{n+1},a_{n},...,a_1]}
\]
and
\[
\beta = \overline{[a_{n+1},a_{n},...,a_1,a_0,a_1,...,a_{n-1},a_n]}
\]
be continued fraction expansions of $\alpha$ and $\beta$ respectively. We define
\[
\left(\begin{matrix}
p&q\\
r&s
\end{matrix}\right)
=
\left(\begin{matrix}
a_1&1\\
1&0
\end{matrix}\right)
...
\left(\begin{matrix}
a_n&1\\
1&0
\end{matrix}\right)
, \ 
\left(\begin{matrix}
P&Q\\
R&S
\end{matrix}\right)
=
\left(\begin{matrix}
a_0&1\\
1&0
\end{matrix}\right)
\left(\begin{matrix}
p&q\\
r&s
\end{matrix}\right)
\left(\begin{matrix}
a_{n+1}&1\\
1&0
\end{matrix}\right)
\left(\begin{matrix}
p&r\\
q&s
\end{matrix}\right)
, \ 
\]
\[
\left(\begin{matrix}
P'&Q'\\
R'&S'
\end{matrix}\right)
=
\left(\begin{matrix}
a_{n+1}&1\\
1&0
\end{matrix}\right)
\left(\begin{matrix}
p&r\\
q&s
\end{matrix}\right)
\left(\begin{matrix}
a_0&1\\
1&0
\end{matrix}\right)
\left(\begin{matrix}
p&q\\
r&s
\end{matrix}\right).
\]\par
Assume that $p$ is even. Then $q$ and $r$ are odd. By $u_{4pq} \equiv 0$ (mod 4), $Q \equiv a_0+a_{n+1}+2s \equiv 0$ (mod 4) holds. Assume that $p$ is odd. When $q$ is even, we have $R \equiv a_{n+1} \equiv 0$ (mod 4). Therefore, $s$ is odd and $Q \equiv a_0+2r \equiv 0$ holds. Similarly, $s$ is odd and $a_{n+1} + 2q \equiv 0$ (mod 4) when $r$ is even and $q$ and $r$ are odd and $a_0 \equiv a_{n+1} \equiv 2$ (mod 4) holds when $s$ is even. \par
In summary, we have
\[
\begin{pmatrix}
p&q\\
r&s
\end{pmatrix}
\equiv
\begin{cases}
\begin{pmatrix}
0&1\\
1&1
\end{pmatrix}
&(a_0 +a_{n+1} \equiv 2\ \mathrm{(mod\ 4)})\\
\begin{pmatrix}
0&1\\
1&0
\end{pmatrix}
&(a_0 +a_{n+1} \equiv 0\ \mathrm{(mod\ 4)})\\
\begin{pmatrix}
1&0\\
0&1
\end{pmatrix}
&(a_0\equiv 0\ \mathrm{(mod\ 4)}, \ a_{n+1} \equiv 0\ \mathrm{(mod\ 4)})\\
\begin{pmatrix}
1&0\\
1&1
\end{pmatrix}
&(a_0 \equiv 2\ \mathrm{(mod\ 4)},\ a_{n+1} \equiv 0\ \mathrm{(mod\ 4)})\\
\begin{pmatrix}
1&1\\
0&1
\end{pmatrix}
&(a_0 \equiv 0\ \mathrm{(mod\ 4)}, \ a_{n+1} \equiv 2\ \mathrm{(mod\ 4)})\\
\begin{pmatrix}
1&1\\
1&0
\end{pmatrix}
&(a_0 \equiv 2\ \mathrm{(mod\ 4)}, \ a_{n+1} \equiv 2\ \mathrm{(mod\ 4)})\\
\end{cases}
\ \mathrm{(mod\ 2)}.
\]
By \cite{hope}, 
\[
p^2 \equiv \left(\frac{a_0}{2}+r\right)^2+(-1)^{n+1}\ \mathrm{(mod\ 4)}
\]
holds. Therefore, $p$ is even if and only if $n$ is even and we have 
\[
S(\alpha) \equiv 2\ \mathrm{(mod\ 4)}
\]
from Lemma \ref{number}. (From this, we get $t_{pq} \equiv 6$ (mod 8) when $u_{pq}$ is even. )
\end{proof}

\begin{thm}\label{pand4p}
Let $p$ be a prime number such that $p \equiv 1$ (mod 4). Then we have
\[
p \equiv 1\ \mathrm{(mod\ 8)} \Rarrow \kappa(p) \equiv \kappa(4p)+2\ \mathrm{(mod\ 4)},
\]
\[
p \equiv 5\ \mathrm{(mod\ 8)} \Rarrow \kappa(p) \equiv \kappa(4p)\ \mathrm{(mod\ 4)}.
\]
\end{thm}

\begin{proof}
Let positive integers $x_p$ and $y_p$ satisfy $p=x_p^2+y_p^2$, where $x_p$ is even and $y_p$ is odd. Let $\alpha$, $\beta$, $\gamma$ and $\delta$ be the largest solutions of equations
\[
x^2-Bx-\frac{1}{4}(p-B^2)=0 \left(B=
\begin{cases}
\lfloor\sqrt{p}\rfloor\ &(\lfloor\sqrt{p}\rfloor \mathrm{:odd})\\
\lfloor\sqrt{p}\rfloor-1\ &(\lfloor\sqrt{p}\rfloor \mathrm{:even})
\end{cases}
\right),
\]
\[
\frac{x_{p}}{2}x^2-y_{p}x-\frac{x_{p}}{2}=0,
\]
\[
x^2-2\left\lfloor\frac{\sqrt{p}}{2}\right\rfloor x-\left(p-{\left\lfloor\frac{\sqrt{p}}{2}\right\rfloor}^2\right)=0,
\]
and
\[
y_px^2-2x_px-y_p=0
\]
respectively. By Lemma \ref{classnumber}, we have 
\[
h(4p)=
\begin{cases}
h(p)\ &(p \equiv 1\ \mathrm{(mod\ 8)\ or}\ p \equiv 5\ \mathrm{(mod\ 8)\ and\ }u_p\mathrm{:odd})\\
3h(p) &(p \equiv 5\ \mathrm{(mod\ 8)\ and\ }u_p\mathrm{:even}).
\end{cases}
\]
Therefore, it suffices to prove that
\[
l(\alpha)\equiv 
\begin{cases}
l(\gamma)\ &\left(u_p\mathrm{:odd} \right)\\
l(\gamma)+2 \ & \left(u_p\mathrm{:even}\right)
\end{cases}
\ \mathrm{(mod\ 4)}.
\]\par
Let
\[
\alpha=\overline{[a_0,a_1,...,a_{n},a_{n},...,a_1]},
\]
\[
\beta = \overline{[a_n,a_{n-1},...,a_1,a_0,a_1,...,a_{n-1},a_n]},
\]
\[
\gamma=\overline{[b_0,b_1,...,b_{m},b_{m},...,b_1]}
\]
and
\[
\delta = \overline{[b_m,b_{m-1},...,b_1,b_0,b_1,...,b_{m-1},b_m]}
\]
be continued fraction expansions of $\alpha$, $\beta$, $\gamma$ and $\delta$ respectively and set
\[
\left(\begin{matrix}
p&q\\
r&s
\end{matrix}\right)
=
\left(\begin{matrix}
a_1&1\\
1&0
\end{matrix}\right)
...
\left(\begin{matrix}
a_n&1\\
1&0
\end{matrix}\right)
, \ 
\left(\begin{matrix}
P&Q\\
R&S
\end{matrix}\right)
=
\left(\begin{matrix}
a_0&1\\
1&0
\end{matrix}\right)
\left(\begin{matrix}
p&q\\
r&s
\end{matrix}\right)
\left(\begin{matrix}
p&r\\
q&s
\end{matrix}\right)
, \ 
\]
\[
\left(\begin{matrix}
P'&Q'\\
R'&S'
\end{matrix}\right)
=
\left(\begin{matrix}
p&r\\
q&s
\end{matrix}\right)
\left(\begin{matrix}
a_{0}&1\\
1&0
\end{matrix}\right)
\left(\begin{matrix}
p&q\\
r&s
\end{matrix}\right)
, \ 
\]
\[
\left(\begin{matrix}
t&u\\
v&w
\end{matrix}\right)
=
\left(\begin{matrix}
b_1&1\\
1&0
\end{matrix}\right)
...
\left(\begin{matrix}
b_m&1\\
1&0
\end{matrix}\right)
, \ 
\left(\begin{matrix}
T&U\\
V&W
\end{matrix}\right)
=
\left(\begin{matrix}
b_0&1\\
1&0
\end{matrix}\right)
\left(\begin{matrix}
t&u\\
v&w
\end{matrix}\right)
\left(\begin{matrix}
t&v\\
u&w
\end{matrix}\right)
, \ 
\]
\[
\left(\begin{matrix}
T'&U'\\
V'&W'
\end{matrix}\right)
=
\left(\begin{matrix}
t&v\\
u&w
\end{matrix}\right)
\left(\begin{matrix}
b_0&1\\
1&0
\end{matrix}\right)
\left(\begin{matrix}
t&u\\
v&w
\end{matrix}\right).
\]\par
We assume that $u_p$ is odd. Since $R = p^2+q^2 \equiv 1$ (mod 4), the parities of $p$ and $q$ are different. When $q$ is even, $S'$ is even and $t_p \equiv y_p$ (mod 4) and $Q' \equiv \frac{u_p-1}{2}+(-1)^n \equiv \frac{x_p}{2}$ hold. From $S=pr+qs= \frac{t_p-a_0u_p}{2}$, when $r$ is even, $t_p \equiv a_0$ (mod 4) holds. Moreover, when $r$ is odd, $t_p \equiv a_0+2$ (mod 4) holds. When $p$ is even, $t_p \equiv -y_p$ (mod 4) and $\frac{u_p-1}{2}+(-1)^n +2 \equiv \frac{x_p}{2}$ (mod 4) hold. When $s$ is even, $t_p \equiv a_0$ (mod 4) holds. Futhermore, when $s$ is odd, $t_p \equiv a_0+2$ (mod 4) holds. Next, we assume that $u_p$ is even. Since $R=p^2+q^2 \equiv 2$ (mod 4), when $r$ is even, $P' \equiv a_0 \equiv \frac{t_p}{2}+y_p$ (mod 4) holds. Similarly, when $s$ is even, $S' \equiv a_0 \equiv \frac{t_{p}}{2}-y_p$ (mod 4) holds. Moreover, since $u_{4p}$ is odd, $V=t^2+u^2 \equiv 1$ (mod 4) holds and the parities of $t$ and $u$ are different. When $u$ is even, $U' \equiv (-1)^m \equiv y_p$ (mod 4) holds. When $v$ is even, $W$ is even and $b_0 \equiv 0$ (mod 4) holds. Futhermore, when $w$ is even, $W$ is odd  and $b_0 \equiv 2$ (mod 4) holds. When $t$ is even, $(-1)^{m+1} \equiv y_p$ (mod 4) holds. When $v$ is even, $b_0 \equiv 2$ (mod 4) holds. Similarly, when $w$ is even, $b_0 \equiv 0$ (mod 4) holds.\par
In summary, we have
\[
\begin{pmatrix}
p&q\\
r&s
\end{pmatrix}
\equiv
\begin{cases}
\begin{pmatrix}
1&0\\
0&1
\end{pmatrix}
(u_p:\mathrm{odd}, \ t_p \equiv y_p \ \mathrm{(mod\ 4)}, \ t_p \equiv a_{0}\ \mathrm{(mod\ 4)})\ &(1)\\
\begin{pmatrix}
1&0\\
1&1
\end{pmatrix}
(u_p:\mathrm{odd}, \ t_p \equiv y_p \ \mathrm{(mod\ 4)}, \ t_p \equiv a_{0}+2\ \mathrm{(mod\ 4)})&(2)\\
\begin{pmatrix}
0&1\\
1&0
\end{pmatrix}
(u_p:\mathrm{odd}, \ t_p \equiv -y_p \ \mathrm{(mod\ 4)}, \ t_p \equiv a_{0}\ \mathrm{(mod\ 4)})&(3)\\
\begin{pmatrix}
0&1\\
1&1
\end{pmatrix}
(u_p:\mathrm{odd}, \ t_p \equiv -y_p \ \mathrm{(mod\ 4)}, \ t_p \equiv a_{0}+2\ \mathrm{(mod\ 4)})&(4)\\
\begin{pmatrix}
1&1\\
0&1
\end{pmatrix}
(u_p:\mathrm{even}, \ a_0 \equiv \frac{t_p}{2}+y_p\ \mathrm{(mod\ 4)})&(5)\\
\begin{pmatrix}
1&1\\
1&0
\end{pmatrix}
(u_p:\mathrm{even}, \ a_0 \equiv \frac{t_p}{2}-y_p\ \mathrm{(mod\ 4)})&(6)\\
\end{cases}
\ \mathrm{(mod\ 2)},
\]
\[
\begin{pmatrix}
t&u\\
v&w
\end{pmatrix}
\equiv
\begin{cases}
\begin{pmatrix}
1&0\\
0&1
\end{pmatrix}
(b_0 \equiv 0 \ \mathrm{(mod\ 4)}, \ y_p \equiv (-1)^{m}\ \mathrm{(mod\ 4)})\ &\mathrm{(i)}\\
\begin{pmatrix}
0&1\\
1&0
\end{pmatrix}
(b_0 \equiv 0 \ \mathrm{(mod\ 4)}, \ y_p \equiv (-1)^{m+1}\ \mathrm{(mod\ 4)})&\mathrm{(ii)}\\
\begin{pmatrix}
1&0\\
1&1
\end{pmatrix}
(b_0 \equiv 2 \ \mathrm{(mod\ 4)}, \ y_p \equiv (-1)^{m}\ \mathrm{(mod\ 4)})&\mathrm{(iii)}\\
\begin{pmatrix}
0&1\\
1&1
\end{pmatrix}
(b_0 \equiv 2 \ \mathrm{(mod\ 4)}, \ y_p \equiv (-1)^{m+1}\ \mathrm{(mod\ 4)})&\mathrm{(iv)}\\
\end{cases}
 \ \mathrm{(mod\ 2)}.
\]\par
First, we discuss the case when $u_p$ is odd. When (1) and (i) holds, $Q' \equiv \frac{u_p-1}{2}+(-1)^n \equiv \frac{x_p}{2}$ (mod 4) and $W' \equiv u_{4p}-1 \equiv \frac{t_{4p}}{2}-x_p$ (mod 8). Since $t_{4p}=(t_{p}^2+3)t_p$ and $u_{4p} = \frac{t_{p}^2+1}{2}u_{p}$, we have $(-1)^n \equiv (-1)^m$ (mod 4) and $l(\alpha) \equiv l(\gamma)$ (mod 4). The other cases can be similarly shown it. \par
Next, we assume that $u_p$ is even. Since $\begin{pmatrix}
P&Q\\
R&S
\end{pmatrix}$
and
$\begin{pmatrix}
T&U\\
V&W
\end{pmatrix}$
 are elements of the automorphism groups of $\alpha$ and $\beta$ respectively, it holds
\[
\begin{pmatrix}
t&u\\
v&w
\end{pmatrix}
=
\begin{cases}
\begin{pmatrix}
p+q&p-q\\
\frac{r+s}{2}&\frac{r-s}{2}
\end{pmatrix}\ &(a_0=\left\lfloor\sqrt{p}\right\rfloor)\\
\begin{pmatrix}
p+q&p-q\\
\frac{r+s}{2}+\frac{p+q}{2}&\frac{r-s}{2}+\frac{p-q}{2}
\end{pmatrix} &(a_0=\left\lfloor\sqrt{p}\right\rfloor-1)
\end{cases}.
\]
Therefore, we have $(-1)^n = (-1)^{m+1}$ and $l(\alpha) \equiv l(\gamma)+2$ (mod 4). 
\end{proof}

\begin{thm}\label{cal8p}
Let $p$ be a prime number such that $p \equiv 1$ (mod 4). Then we have 
\[
p \equiv 1\ \mathrm{(mod\ 8)} \Rarrow \kappa(8p) \equiv 2\ \mathrm{(mod\ 4)},
\]
\[
p \equiv 5\ \mathrm{(mod\ 8)} \Rarrow \kappa(8p) \equiv 0\ \mathrm{(mod\ 4)}.
\]
\end{thm}

\begin{proof}
Let $x_p$ and $y_p$ be integers satisfying $p=x_p^2+y_p^2$ and $0<x_p<y_p$ and $\alpha$, $\beta$, $\gamma$ and $\delta$ be the largest solutions of equations
\[
(x_p+y_p)x^2+2(x_p-y_p)x-(x_p+y_p)=0,
\]
\[
 (y_p-x_p)x^2-2(x_p+y_p)x+(x_p-y_p)=0,
\]
\[
x^2-2\lfloor\sqrt{2p}\rfloor x+\lfloor\sqrt{2p}\rfloor^2-2p=0,
\]
and
\[
 2x^2-4\left\lfloor\sqrt{\frac{p}{2}}\right\rfloor x+2\left\lfloor\sqrt{\frac{p}{2}}\right\rfloor^2-p=0
\]
respectively. It suffices to prove that 
\[
l(\alpha)+l(\beta)\equiv 2\ \mathrm{(mod\ 4)}\quad (N(\varepsilon_{8p})=-1),
\]
\[
l(\alpha)+l(\gamma)\equiv 2\ \mathrm{(mod\ 4)}\quad (N(\varepsilon_{8p})=1).
\]\par
When $N(\varepsilon_{8p})=-1$ holds, let
\[
\alpha=\overline{[a_0,a_1,...,a_{n},a_{n},...,a_1]}
\]
and
\[
\beta = \overline{[b_0,b_1,...,b_{m},b_{m},...,b_1]}
\]
be continued fraction expansions of $\alpha$ and $\beta$ respectively and set 
\[
\left(\begin{matrix}
p&q\\
r&s
\end{matrix}\right)
=
\left(\begin{matrix}
a_1&1\\
1&0
\end{matrix}\right)
...
\left(\begin{matrix}
a_n&1\\
1&0
\end{matrix}\right)
, \ 
\left(\begin{matrix}
P&Q\\
R&S
\end{matrix}\right)
=
\left(\begin{matrix}
a_0&1\\
1&0
\end{matrix}\right)
\left(\begin{matrix}
p&q\\
r&s
\end{matrix}\right)
\left(\begin{matrix}
p&r\\
q&s
\end{matrix}\right)
\]
\[
\left(\begin{matrix}
t&u\\
v&w
\end{matrix}\right)
=
\left(\begin{matrix}
b_1&1\\
1&0
\end{matrix}\right)
...
\left(\begin{matrix}
b_m&1\\
1&0
\end{matrix}\right)
, \ 
\left(\begin{matrix}
T&U\\
V&W
\end{matrix}\right)
=
\left(\begin{matrix}
b_0&1\\
1&0
\end{matrix}\right)
\left(\begin{matrix}
t&u\\
v&w
\end{matrix}\right)
\left(\begin{matrix}
t&v\\
u&w
\end{matrix}\right) .
\]
According to \cite{mori}, since 
$\begin{pmatrix}
P&Q\\
R&S
\end{pmatrix}\equiv
\begin{pmatrix}
T&U\\
V&W
\end{pmatrix}$ (mod 4), we get, when 
$\begin{pmatrix}
P&Q\\
R&S
\end{pmatrix}\equiv
\begin{pmatrix}
0&1\\
1&2
\end{pmatrix}$ (mod 4), $\biggl(\begin{pmatrix}
t&u\\
v&w
\end{pmatrix}$ also satisfies the same condition when 
$\begin{pmatrix}
T&U\\
V&W
\end{pmatrix}\equiv
\begin{pmatrix}
0&1\\
1&2
\end{pmatrix}$ (mod 4)$.\biggr)$
\[
\begin{pmatrix}
p&q\\
r&s
\end{pmatrix}
\equiv
\begin{cases}
\begin{pmatrix}
1&1\\
1&0
\end{pmatrix}
(n\mathrm{:even}, \ a_0 \equiv 2\ \mathrm{(mod\ 4)})\\
\begin{pmatrix}
0&1\\
1&0
\end{pmatrix}
(n\mathrm{:odd}, \ a_0 \equiv 2\ \mathrm{(mod\ 4)})\\
\begin{pmatrix}
1&0\\
1&1
\end{pmatrix}
(n\mathrm{:even}, \ a_0 \equiv 0\ \mathrm{(mod\ 4)})\\
\begin{pmatrix}
0&1\\
1&1
\end{pmatrix}
(n\mathrm{:odd}, \ a_0 \equiv 0\ \mathrm{(mod\ 4)})
\end{cases}\mathrm{(mod\ 2)}.
\]
Since $R' =p^2+r^2=u_{8p}$ and $V'=t^2+v^2=2u_{8p}$ or $R'=2u_{8p}$ and $V'=u_{8p}$ holds, we get $l(\alpha)+l(\beta) \equiv 2$ (mod 4). We can find $l(\alpha)+l(\beta) \equiv 2$ (mod 4) in the same way as in the other cases. \par
When $N(\varepsilon_{8p})=1$ (from which $p \equiv 1$ (mod 8) follows), we have from the proof of Proposition \ref{8p}
\[
l(\alpha)\equiv
\begin{cases}
2\ &(u_{8p} \equiv 2\ \mathrm{(mod\ 4)})\\
0\ &(u_{8p} \equiv 0\ \mathrm{(mod\ 4)})\\
\end{cases}
\]
and
\[
l(\gamma)\equiv
\begin{cases}
0\ &(u_{8p} \equiv 2\ \mathrm{(mod\ 4)})\\
2\ &(u_{8p} \equiv 0\ \mathrm{(mod\ 4)})\\
\end{cases}.
\]
Therefore, we get $l(\alpha)+l(\gamma)\equiv 2$ (mod 4).

\end{proof}

\begin{thm}\label{4pand8p}
Let $p$ be a prime number such that $p \equiv 3$ (mod 4). Then we have
\[
p \equiv 3\ \mathrm{(mod\ 8)} \Rarrow \kappa(4p) \equiv \kappa(8p) \equiv 2 \ \mathrm{(mod\ 4)},
\]
\[
p \equiv 7\ \mathrm{(mod\ 8)} \Rarrow \kappa(4p) \equiv \kappa(8p) \equiv 0 \ \mathrm{(mod\ 4)}.
\]
\end{thm}

\begin{proof}
Let $D$ stand for $4p$ or $8p$ and $\alpha$ and $\beta$ be the largest solutions of equations
\[
x^2-2\left\lfloor\sqrt{\frac{D}{4}}\right\rfloor x-\left(\frac{D}{4}-\left\lfloor\sqrt{\frac{D}{4}}\right\rfloor^2\right)=0
\]
and
\[
\begin{cases}
2x^2-2Bx-\frac{1}{2}(p-B^2)=0 \left(B=
\begin{cases}
\left\lfloor\sqrt{p}\right\rfloor\ &(\lfloor\sqrt{p}\rfloor \mathrm{:odd})\\
\left\lfloor\sqrt{p}\right\rfloor-1\ &(\lfloor\sqrt{p}\rfloor \mathrm{:even})
\end{cases}
\right)&(D=4p)\\
2x^2-4\left\lfloor\sqrt{\frac{p}{2}}\right\rfloor x-(p-2\left\lfloor\sqrt{\frac{p}{2}}\right\rfloor^2)=0&(D=8p)
\end{cases}
\]
respectively. It suffices to prove that 
\[
l(\alpha) \equiv 
\begin{cases}
2\ & (p \equiv 3\ \mathrm{(mod\ 8)})\\
0\ &(p \equiv 7\ \mathrm{(mod\ 8)}) 
\end{cases}
\ \mathrm{(mod\ 4)}.
\]\par
Let
\[
\alpha=\overline{[a_0,a_1,...,a_{n},a_{n+1},a_{n},...,a_1]}
\]
and
\[
\beta = \overline{[a_{n+1},a_{n},...,a_1,a_0,a_1,...,a_{n-1},a_n]}
\]
be continued fraction expansions of $\alpha$ and $\beta$ respectively. Let $\omega_i$ be an element of $\mathcal{Q}(D)$ such that
\[
\omega_i=\overline{[a_i,a_{i-1},...,a_{1},a_{0},a_{1},...,a_{n},a_{n+1},a_{n},...,a_{i+1}]}.
\] 
Then we obtain
\begin{align*}
2[1,\beta]
&=2[1,\omega_{n}^{-1}]\\
&=2\omega_{n}^{-1}[1,\omega_{n}]\\
&=...\\
&=2\omega_{n}^{-1}...\omega_0^{-1}[1,\alpha].
\end{align*}
Since the norms of $[1,\alpha]$ and $[1,\beta]$ in $\Z\left[\sqrt{\frac{D}{4}}\right]$ are $1$ and $\frac{1}{2}$ respectively, $N(2\omega_0^{-1}...,\omega_{n}^{-1})=\pm2$. Since $N(\omega_i)$ is negative, if $N(2\omega_0^{-1}...,\omega_{n}^{-1})=2$ holds, $n$ is odd and if $N(2\omega_0^{-1}...,\omega_{n}^{-1})=-2$ holds, $n$ is even. We put $x+y\sqrt{\frac{D}{4}}=2\omega_0^{-1}...,\omega_{n}^{-1}$. When $N\left(x+y\sqrt{\frac{D}{4}}\right)=2$ holds, since $x^2-\frac{D}{4}y^2=2$, we get $\left(\frac{2}{p}\right)=(-1)^{\frac{p^2-1}{8}}=1$ and $p \equiv 7$ (mod 8). Similarly, we have $p \equiv 3$ (mod 8) when $N\left(x+y\sqrt{\frac{D}{4}}\right)=-2$. Therefore, we have
\[
l(\alpha) \equiv 
\begin{cases}
2\ & (p \equiv 3\ \mathrm{(mod\ 8)})\\
0\ &(p \equiv 7\ \mathrm{(mod\ 8)}) 
\end{cases}
\ \mathrm{(mod\ 4)}.
\]
\end{proof}

When we take $\alpha=\sqrt{\frac{D}{4}}+\left\lfloor\sqrt{\frac{D}{4}}\right\rfloor$ in the above arguments, we obtain the following corollary.

\begin{cor}
Let $p$ be a prime number such that $p \equiv 3$ (mod 4). Then we have
\[
p \equiv 3\ (\mathrm{mod\ 8}) \Rarrow l(\sqrt{p}) \equiv l(\sqrt{2p}) \equiv 2\ (\mathrm{mod\ 4}),
\]
\[
p \equiv 7\ (\mathrm{mod\ 8}) \Rarrow l(\sqrt{p}) \equiv l(\sqrt{2p}) \equiv 0\ (\mathrm{mod\ 4}).
\]
\end{cor}

\begin{thm}\label{9p}
Let $p$ be a prime number such that $p \equiv 1$ (mod 4). Then we have
\[
p=5\ \mathrm{or}\ p \equiv 1\ \mathrm{(mod\ 3)} \Rarrow \kappa(9p) \equiv 2\ \mathrm{(mod\ 4)},
\]
\[
p \equiv 2\ \mathrm{(mod\ 3)}\ (p \not=5)\ \Rarrow \kappa(9p) \equiv 0\ \mathrm{(mod\ 4)}.
\]
\end{thm}

\begin{proof}
Since $\kappa(45)=2$ holds, in the following, we may assume $p \not= 5$.
Let $\alpha$ and $\beta$ be the largest solution of equations
\[
x^2-Bx-\frac{1}{4}(9p-B^2)=0 \left(B=
\begin{cases}
\lfloor3\sqrt{p}\rfloor\ &(\lfloor3\sqrt{p}\rfloor \mathrm{:odd})\\
\lfloor3\sqrt{p}\rfloor-1\ &(\lfloor3\sqrt{p}\rfloor \mathrm{:even})
\end{cases}
\right)
\]
and
\begin{spacing}{1.3}
\begin{center}
$9x^2-9B'x-\frac{1}{4}(p-9{B'}^2)=0 \left(B'=
\begin{cases}
\left\lfloor\frac{\sqrt{p}}{3}\right\rfloor\ &\left(\left\lfloor\frac{\sqrt{p}}{3}\right\rfloor \mathrm{:odd}\right)\\
\left\lfloor\frac{\sqrt{p}}{3}\right\rfloor-1\ &\left(\left\lfloor\frac{\sqrt{p}}{3}\right\rfloor \mathrm{:even}\right)
\end{cases}
\right)$
\end{center}
\end{spacing}
\noindent respectively. It suffices to prove that 
\[
l(\alpha) \equiv 
\begin{cases}
2\ &(p \equiv 1\ \mathrm{(mod\ 3}))\\
0\ & (p \equiv 2\ \mathrm{(mod\ 3)})
\end{cases}
\ \mathrm{(mod\ 4)}.
\]\par
Let
\[
\alpha=\overline{[a_0,a_1,...,a_{n},a_{n+1},a_{n},...,a_1]}
\]
and
\[
\beta = \overline{[a_{n+1},a_{n},...,a_1,a_0,a_1,...,a_{n-1},a_n]}
\]
be continued fraction expansions of $\alpha$ and $\beta$ respectively. Then we get
\begin{align*}
9[1,\beta]
&=9[1,\omega_{n}^{-1}]\\
&=9\omega_{n}^{-1}[1,\omega_{n-1}]\\
&=...\\
&=9\omega_{n}^{-1}...\omega_0^{-1}[1,\alpha].
\end{align*}
Since the norms of $[1,\alpha]$ and $[1,\beta]$ in $\Z\left[\frac{1+3\sqrt{p}}{2}\right]$are $1$ and $\frac{1}{9}$ respectively, $N(9\omega_0^{-1}...,\omega_{n}^{-1})=\pm9$. When $N\left(\frac{x+3y\sqrt{p}}{2}\right)=9$ holds, we get $x^2-9py^2=36$. By $\left[9,\frac{9+3\sqrt{p}}{2}\right]=\left[\frac{x+3y\sqrt{p}}{2},\frac{x+9yp+3(x+y)\sqrt{p}}{4}\right]$, $y$ is odd and we have $p \equiv 2$ (mod 3). Similarly, we have $p \equiv 1$ (mod 3) when $N\left(\frac{x+y\sqrt{3p}}{2}\right)=-3$. Therefore, 
\[
l(\alpha) \equiv 
\begin{cases}
2\ &(p \equiv 1\ \mathrm{(mod\ 3}))\\
0\ & (p \equiv 2\ \mathrm{(mod\ 3)})
\end{cases}
\ \mathrm{(mod\ 4)}
\]
holds.
\end{proof}

Since $l\left(\frac{1+\sqrt{45}}{2}\right)=6$ holds, we get the following corollary. 

\begin{cor}
Let $p$ be a prime number such that $p \equiv 1$ (mod 4). Then we have
\[
p=5\  \mathrm{or}\  p \equiv 1\ (\mathrm{mod\ 3}) \Rarrow l\left(\frac{1+\sqrt{3p}}{2}\right) \equiv 2\ (\mathrm{mod\ 4}),
\]
\[
p \equiv 2\ (\mathrm{mod\ 3}) \ (p \not=5)\Rarrow l\left(\frac{1+\sqrt{3p}}{2}\right) \equiv 0\ (\mathrm{mod\ 4}).
\]
\end{cor}

\begin{thm}\label{16p}
Let $p>2$ be a prime number. Then we have
\[
p=3\ \mathrm{or}\ p \equiv 1\ \mathrm{(mod\ 4)} \Rarrow \kappa(16p) \equiv 2\ \mathrm{(mod\ 4)},
\]
\[
p \equiv 3\ \mathrm{(mod\ 4)} \ (p \not=3)\ \Rarrow \kappa(16p) \equiv 0\ \mathrm{(mod\ 4)}.
\]
\end{thm}

\begin{proof}
Since $\kappa(48)=2$ holds, in the following, we may suppose $p \not=3$.\\
Let $\alpha$ and $\beta$ be the largest solutions of equations
\[
x^2-2\lfloor2\sqrt{p}\rfloor x-(4p-\lfloor2\sqrt{p}\rfloor^2)=0
\]
and
\[
4x^2-8\left\lfloor\sqrt{\frac{p}{2}}\right\rfloor x-\left(p-4\left\lfloor\sqrt{\frac{p}{2}}\right\rfloor^2\right)=0
\]
respectively. It suffices to prove that
\[
l(\alpha) \equiv 
\begin{cases}
2\ &(p \equiv 1\ \mathrm{(mod\ 4}))\\
0\ & (p \equiv 3\ \mathrm{(mod\ 4)})
\end{cases}
\ \mathrm{(mod\ 4)}.
\]\par
Let
\[
\alpha=\overline{[a_0,a_1,...,a_{n},a_{n+1},a_{n},...,a_1]}
\]
and
\[
\beta = \overline{[a_{n+1},a_{n},...,a_1,a_0,a_1,...,a_{n-1},a_n]}
\]
be continued fraction expansions of $\alpha$ and $\beta$ respectively. Then we get
\begin{align*}
4[1,\beta]
&=4[1,\omega_{n}^{-1}]\\
&=4\omega_{n}^{-1}[1,\omega_{n-1}]\\
&=...\\
&=4\omega_{n}^{-1}...\omega_0^{-1}[1,\alpha].
\end{align*}
Since the norms of $[1,\alpha]$ and $[1,\beta]$ in $\Z[2\sqrt{p}]$ are $1$ and $\frac{1}{4}$ respectively, $N(4\omega_0^{-1}...,\omega_{n}^{-1})=\pm4$. When $N(x+2y\sqrt{p})=4$ holds, since $[4,2\sqrt{p}]=[x+2y\sqrt{p},4py+2x\sqrt{p}]$, $x \equiv 0$ (mod 4) holds. Therfore, we get $p \equiv 3$ (mod 4). Similarly, we have $p \equiv 1$ (mod 4) when $N(x+2y\sqrt{p})=-4$. Therefore, 
\[
l(\alpha) \equiv 
\begin{cases}
2\ &(p \equiv 1\ \mathrm{(mod\ 4}))\\
0\ & (p \equiv 3\ \mathrm{(mod\ 4)})
\end{cases}
\ \mathrm{(mod\ 4)}
\]
holds.
\end{proof}

Since $l(2\sqrt{3})=2$ holds, we get the next corollary.
\begin{cor}
Let $p>2$ be a prime number. Then we have
\[
p=3\  \mathrm{or}\  p \equiv 1\ (\mathrm{mod\ 4}) \Rarrow l(2\sqrt{p}) \equiv 2\ (\mathrm{mod\ 4}),
\]
\[
p \equiv 3\ (\mathrm{mod\ 4}) \ (p \not=3)\Rarrow l(2\sqrt{p}) \equiv 0\ (\mathrm{mod\ 4}).
\]
\end{cor}


\begin{thm}\label{calpq}
Let $p$ and $q$ ($p<q$) be prime numbers such that $p \equiv q \equiv 3$ (mod 4). Then we have 
\[
\kappa(pq) \equiv 1+\left(\frac{q}{p}\right)\ \mathrm{(mod\ 4)}.
\]
\end{thm}
\begin{proof}
Let $\alpha$ and $\beta$ be the largest solutions of equations
\[
x^2-Bx-\frac{1}{4}(pq-B^2)=0 \left(B=
\begin{cases}
\lfloor\sqrt{pq}\rfloor\ &(\lfloor\sqrt{pq}\rfloor \mathrm{:odd})\\
\lfloor\sqrt{pq}\rfloor-1\ &(\lfloor\sqrt{pq}\rfloor \mathrm{:even})
\end{cases}
\right)
\]
and
\[
px^2-pB'x-\frac{1}{4}(q-p{B'}^2)=0\left(B'=
\begin{cases}
\left\lfloor\sqrt{\frac{q}{p}}\right\rfloor\ &\left(\left\lfloor\sqrt{\frac{q}{p}}\right\rfloor \mathrm{:odd}\right)\\
\left\lfloor\sqrt{\frac{q}{p}}\right\rfloor-1\ &\left(\left\lfloor\sqrt{\frac{q}{p}}\right\rfloor \mathrm{:even}\right)
\end{cases}
\right)
\]
respectively.
It suffices to prove that
\[
l(\alpha) \equiv 1+\left(\frac{q}{p}\right)\ \mathrm{(mod\ 4)}.
\]\par
Let
\[
\alpha=\overline{[a_0,a_1,...,a_{n},a_{n+1},a_{n},...,a_1]}
\]
and
\[
\beta = \overline{[a_{n+1},a_{n},...,a_1,a_0,a_1,...,a_{n-1},a_n]}
\]
be continued fraction expansions of $\alpha$ and $\beta$ respectively. Then we get
\begin{align*}
p[1,\beta]
&=p[1,\omega_{n-1}^{-1}]\\
&=p\omega_{n-1}[1,\omega_{n-1}]\\
&=...\\
&=p\omega_{n-1}...\omega_0[1,\alpha].
\end{align*}
Since the norms of $[1,\alpha]$ and $[1,\beta]$ in $\Z\left[\frac{1+\sqrt{pq}}{2}\right]$ are $1$ and $\frac{1}{p}$ respectively, $N(p\omega_0^{-1}...,\omega_{n}^{-1})=\pm p$. When $N\left(\frac{x+y\sqrt{pq}}{2}\right)=p$ holds, since $x^2-pqy^2=4p$, $\left(\frac{-q}{p}\right)=1$ and  $\left(\frac{q}{p}\right)=-1$ holds. Similarly, we have $\left(\frac{q}{p}\right)=1$ when $N\left(\frac{x+y\sqrt{pq}}{2}\right)=-p$. Therefore, 
\[
l(\alpha) \equiv 1+\left(\frac{q}{p}\right)\ \mathrm{(mod\ 4)}
\]
holds.\\
\end{proof}

\begin{cor}
Let $p$ and $q$ ($p<q$) be prime numbers such that $p \equiv q \equiv 3$ (mod 4). Then we have
\[
l\left(\frac{1+\sqrt{pq}}{2}\right) \equiv 1+\left(\frac{q}{p}\right)\ \mathrm{(mod\ 4)}.
\]
\end{cor}

\begin{thm}\label{cal4pq}
Let $p$ and $q$ ($p<q$) be prime numbers such that $p \equiv q \equiv 3$ (mod 4). Then we have 
\[
\kappa(4pq) \equiv 1+\left(\frac{q}{p}\right)\ \mathrm{(mod\ 4)}.
\]
\end{thm}
\begin{proof}
Let $\alpha$ and $\beta$ be the largest solutions of equations
\[
x^2-2\lfloor\sqrt{pq}\rfloor x+\lfloor\sqrt{pq}\rfloor^2-pq=0
\]
and
\[
px^2-2p\left\lfloor\sqrt{\frac{q}{p}}\right\rfloor x-p\left\lfloor\sqrt{\frac{q}{p}}\right\rfloor^2-q=0
\]
respectively. It suffices to prove that 
\[
l(\alpha) \equiv 1+\left(\frac{q}{p}\right)\ \mathrm{(mod\ 4)}.
\]\par
Let
\[
\alpha=\overline{[a_0,a_1,...,a_{n},a_{n+1},a_{n},...,a_1]}
\]
and
\[
\beta = \overline{[a_{n+1},a_{n},...,a_1,a_0,a_1,...,a_{n-1},a_n]}
\]
be continued fraction expansions of $\alpha$ and $\beta$ respectively. Then we get
\begin{align*}
p[1,\beta]
&=p[1,\omega_{n-1}^{-1}]\\
&=p\omega_{n-1}[1,\omega_{n-1}]\\
&=...\\
&=p\omega_{n-1}...\omega_0[1,\alpha].
\end{align*}
Since the norms of $[1,\alpha]$ and $[1,\beta]$ in $\Z[\sqrt{pq}]$ are $1$ and $\frac{1}{p}$ respectively, $N(p\omega_0^{-1}...,\omega_{n}^{-1})=\pm p$. When $N(x+y\sqrt{pq})=p$ holds, since $x^2-pqy^2=p$, $\left(\frac{-q}{p}\right)=1$ and  $\left(\frac{q}{p}\right)=-1$ holds. Similarly, we have $\left(\frac{q}{p}\right)=1$ when $N(x+y\sqrt{pq})=-p$. Therefore, 
\[
l(\alpha) \equiv 1+\left(\frac{q}{p}\right)\ \mathrm{(mod\ 4)}
\]
holds.
\end{proof}

\begin{cor}
Let $p$ and $q$ ($p<q$) be prime numbers such that $p \equiv q \equiv 3$ (mod 4). Then we have
\[
l(\sqrt{pq}) \equiv 1+\left(\frac{q}{p}\right)\ \mathrm{(mod\ 4)}.
\]
\end{cor}

\begin{thm}\label{4pq}
Let $p$ and $q$ be prime numbers such that $p \equiv 1$ (mod 4) and $q \equiv 3$ (mod 4). Then we have 
\[
\kappa(4pq) \equiv 0\ \mathrm{(mod\ 4)}.
\]
\end{thm}

\begin{proof}
Let $\alpha$, $\beta$, $\gamma$ and $\delta$ be the largest solutions of equations
\[
x^2-2\lfloor\sqrt{pq}\rfloor x-(pq-\lfloor\sqrt{pq}\rfloor^2)=0,
\]
\[
2x^2-2Bx-\frac{1}{2}(pq-B^2)=0 \left(B=
\begin{cases}
\lfloor\sqrt{pq}\rfloor\ &(\lfloor\sqrt{pq}\rfloor \mathrm{:odd})\\
\lfloor\sqrt{pq}\rfloor-1\ &(\lfloor\sqrt{pq}\rfloor \mathrm{:even})
\end{cases}
\right),
\]
\[
\begin{cases}
px^2-2p\left\lfloor\sqrt{\frac{q}{p}}\right\rfloor x-\left(q-p{\left\lfloor\sqrt{\frac{q}{p}}\right\rfloor}^2\right)=0\\
2px^2-2pB'x-\frac{1}{2}(q-p{B'}^2)=0\left(B'=
\begin{cases}
\left\lfloor\sqrt{\frac{q}{p}}\right\rfloor\ &\left(\left\lfloor\sqrt{\frac{q}{p}}\right\rfloor \mathrm{:odd}\right)\\
\left\lfloor\sqrt{\frac{q}{p}}\right\rfloor-1\ &\left(\left\lfloor\sqrt{\frac{q}{p}}\right\rfloor \mathrm{:even}\right)
\end{cases}
\right)
\end{cases}
(p<q),
\]
and
\[
\begin{cases}
qx^2-2q\left\lfloor\sqrt{\frac{p}{q}}\right\rfloor x-\left(p-q{\left\lfloor\sqrt{\frac{p}{q}}\right\rfloor}^2\right)=0\\
2qx^2-2qB'x-\frac{1}{2}(p-q{B'}^2)=0\left(B'=
\begin{cases}
\left\lfloor\sqrt{\frac{p}{q}}\right\rfloor\ &\left(\left\lfloor\sqrt{\frac{p}{q}}\right\rfloor \mathrm{:odd}\right)\\
\left\lfloor\sqrt{\frac{p}{q}}\right\rfloor-1\ &\left(\left\lfloor\sqrt{\frac{p}{q}}\right\rfloor \mathrm{:even}\right)
\end{cases}
\right)
\end{cases}
(p>q)
\]
respectively.  From now on, we asume that $p<q$, $\alpha \sim \beta$ and $\gamma \sim \delta$. (The proof can be similarly extended to other cases.) 
It suffices to prove that
\[
l(\alpha) + l(\gamma) \equiv 0\ \mathrm{(mod\ 4)}.
\]\par
Let
\[
\alpha=\overline{[a_0,a_1,...,a_{n},a_{n+1},a_{n},...,a_1]},
\]
\[
\beta = \overline{[a_{n+1},a_{n},...,a_1,a_0,a_1,...,a_{n}]},
\]
\[
\gamma=\overline{[b_0,b_1,...,b_{m},b_{m+1},b_{m},...,b_1]}
\]
and
\[
\delta = \overline{[b_{m+1},b_{m},...,b_1,b_0,b_1,...,b_{m}]}
\]
be continued fraction expansions of $\alpha$, $\beta$, $\gamma$ and $\delta$ respectively. Then we get
\begin{align*}
2[1,\beta]
&=2[1,{\omega}_{n}^{-1}]\\
&=2{\omega}_{n}^{-1}[1,\omega_{n}]\\
&=...\\
&=2{\omega}_{n}^{-1}...\omega_0^{-1}[1,\alpha]
\end{align*}
and 
\begin{align*}
2p[1,\delta]
&=2p[1,{\omega'}_{m}^{-1}]\\
&=2p{\omega'}_{m}^{-1}[1,\omega'_{m}]\\
&=...\\
&=2p{\omega'}_{m}^{-1}...{\omega'}_0^{-1}[1,\gamma].
\end{align*}
Since the norms of $[1,\alpha]$, $[1,\beta]$, $[1,\gamma]$ and $[1,\delta]$ in $\Z[\sqrt{pq}]$ are $1$, $\frac{1}{2}$, $\frac{1}{p}$ and $\frac{1}{2p}$ respectively, $N(2\omega_0^{-1}...,\omega_{n}^{-1})=\pm2$ and $N(2{\omega'}_0^{-1}...,{\omega'}_{m}^{-1})=\pm2$. When $N(x+y\sqrt{pq})=2$ holds, since $x^2-pqy^2=2$, $\left(\frac{2}{p}\right)=\left(\frac{2}{q}\right)=1$ and $p \equiv 1$ (mod 8)  and $q \equiv 7$ (mod 8) holds. Similarly, we have $p \equiv 1$ (mod 8) and $q \equiv 3$ (mod 8) when $N(x+y\sqrt{pq})=-2$. Therefore, 
\[
l(\alpha) + l(\gamma) \equiv 0\ \mathrm{(mod\ 4)}
\]
holds.

\end{proof}

\section*{ \S 4. Conjecture}
\setcounter{section}{4}
\setcounter{subsection}{0}

For prime numbers $p$ and $q$ congruent to 1 modulo 4, let the positive integers $x_p$, $y_p$, $x_q$ and $y_q$ be determined (uniquely) by
\[
p=x_p^2+y_p^2\ (0<x_p<y_p),\  q=x_q^2+y_q^2\ (0<x_q<y_q).
\]
In \cite{mori}, it is conjectured that the following proposition holds true.
\begin{conj}
Let $p$ and $q$ $(p<q)$ be prime numbers such that $p \equiv q \equiv 1$ (mod 4). We assume $x_p \not\equiv x_q$ (mod 2). Then we have
\[
\kappa^+(pq) \equiv 1-(-1)^{x_p}\left(\frac{q}{p}\right)\ \mathrm{(mod\ 4)}.
\]
\end{conj}
In the case of $N(\varepsilon_{pq})=-1$, from \cite{mori} this conjecture holds true if the two conditinos $l\left(\frac{1+\sqrt{pq}}{2}\right) \equiv l\left(\frac{1+\sqrt{\frac{q}{p}}}{2}\right)$ (mod 4) and $x_py_q-y_px_q<0$ are equivalent. In the case of $N(\varepsilon_{pq})=1$, following the proof of $\kappa^{+}(8p)$, if the conditions $\kappa(pq) \equiv 0$ (mod 4) and $x_py_q-y_px_q<0$ are equivalent, this conjecture holds true.

\end{document}